\documentclass[11pt,twoside, final]{amsart}

\usepackage{amsmath,amsthm,amscd,amsfonts,amssymb,enumerate}
\usepackage{graphicx}		
\usepackage{color}
\usepackage{mathrsfs,url}
\usepackage[colorlinks=true,allcolors=blue]{hyperref}

\newtheorem{theorem}{Theorem}[section]

\theoremstyle{definition}

\theoremstyle{remark}

\numberwithin{equation}{section}

\begin{document}

\title[EXACT CURVED SURFACE AREA OF A FRUSTUM OF HEMIELLIPSOID]
{ANALYTICAL EXPRESSION FOR THE EXACT CURVED SURFACE AREA OF A FRUSTUM OF HEMIELLIPSOID, THROUGH HYPERGEOMETRIC FUNCTION APPROACH}

\author[M.A. Pathan]{M.A. Pathan}
\address[M.A. Pathan]{Centre for Mathematical and Statistical Sciences (CMSS), Peechi,
Thrissur, Kerala-680653, India\newline
Department of Mathematics, Aligarh Muslim University,
Aligarh, U.P., India}
\email{mapathan@gmail.com}

\author[M. I. Qureshi]{Mohd  Idris  Qureshi}
\address[M. I. Qureshi]{Department of Applied Sciences and Humanities,
Faculty of Engineering and Technology, Jamia Millia Islamia (A Central University), New Delhi 110025, India}
\email{miqureshi\_delhi@yahoo.co.in}

\author[J. Majid]{Javid Majid$^*$}
\address[J. Majid]{Department of Applied Sciences and Humanities,
Faculty of Engineering and Technology, Jamia Millia Islamia (A Central University), New Delhi 110025, India}
\email{javidmajid375@gmail.com}

  \thanks{$^*$Corresponding author}



\begin{abstract} Our present investigation is motivated essentially by several interesting applications of generalized hypergeometric functions of one, two and more variables. The
hypergeometric functions are potentially useful and have widespread applications related to the problems in the mathematical, physical, engineering and statistical
sciences. In this article, we aim at obtaining the analytical expression (not previously found and recorded in the literature) for the exact curved surface area of a frustum of hemiellipsoid in terms of Appell's function of second kind and general triple hypergeometric series of Srivastava. The derivation is based on Mellin-Barnes type contour integral representations of generalized hypergeometric function$~_pF_q(z)$, Meijer's $G$-function and series manipulation technique. The closed form for the exact curved surface area of a frustum of hemiellipsoid is also verified numerically by using {\it Mathematica Program}.
\end{abstract}

 \maketitle

\section{Introduction and preliminaries}\label{IP}

\noindent
For the definition of Pochhammer symbols, power series form of generalized hypergeometric function ${}_pF_q(z)$ and several related results, we refer the beautiful monographs (see, e.g., \cite{And1999, Erd1, Lebedev1972, Rain, Slater1966, SriMano})

\begin{theorem} When $p=0,1,2,...$, then
\begin{equation*}
\frac{1}{\Gamma(\bf{-p})}~_4F_3\left[\begin{array}{ll}
A,B,C,D;\\
& z\\
~E,G,\bf{-p};\end{array}\right]=\frac{(A)_{p+1}(B)_{p+1}(C)_{p+1}(D)_{p+1}z^{p+1}}{(E)_{p+1}(G)_{p+1}(p+1)!}\times
\end{equation*}
\begin{equation}\label{A13.50}
\times~_4F_3\left[\begin{array}{ll}
A+p+1,B+p+1,C+p+1,D+p+1;\\
& z\\
~~~~~~~~~~~~~~~~~~E+p+1,~G+p+1,~2+p;\end{array}\right];~|z|<1.
\end{equation}
\begin{proof} Consider the left hand side of the assertion (\ref{A13.50})
\begin{equation*}
\frac{1}{\Gamma(-p)}~_4F_3\left[\begin{array}{ll}
A,B,C,D;\\
& z\\
~~E,G,-p;\end{array}\right]=\sum_{r=0}^{\infty}\frac{(A)_{r}(B)_{r}(C)_{r}(D)_{r}z^{r}}{(E)_{r}(G)_{r}~\Gamma(-p+r)~r!}
\end{equation*}
\begin{equation*}
=\sum_{r=p+1}^{\infty}\frac{(A)_{r}(B)_{r}(C)_{r}(D)_{r}z^{r}}{(E)_{r}(G)_{r}~\Gamma(-p+r)~r!}.
\end{equation*}
Replacing $r$ by $r+p+1$, after simplification and expressing the result in the form of generalized hypergeometric function, we arrive at the result (\ref{A13.50}).
\end{proof}
\end{theorem}

\noindent
Appell's function of second kind \cite[p.53, Eq.(5)]{SriMano} is defined as:
\begin{equation*}
F_2\left[\begin{array}{ll}
a;~b,c;~d,g;~x,y\end{array}\right]=F^{1:1;1}_{0:1;1}\left[\begin{array}{ll}
~a:~b;~c;\\
& x,y\\
-:~d;~g;\\\end{array}\right]~~~~~~~~~~~~~~~~~~~~~~~~~~~~~~~~~~~~
\end{equation*}
\begin{equation}\label{E12.19}
=\sum_{m,n=0}^{\infty}\frac{(a)_{m+n}(b)_m(c)_n~x^m~y^n}{(d)_m (g)_n~ m! ~n!}=\sum_{n=0}^{\infty}\frac{(a)_{n}(c)_n~y^n}{(g)_n ~n!}~{}_2F_1\left[\begin{array}{ll}
a+n,~b;\\
& x\\
~~~~~~~~~d;\end{array}\right].
\end{equation}
Convergence conditions of Appell's double series $F_2$:
\begin{enumerate}
\vsize.1cm
\item[{(i)}] Appell's function $F_2$ is convergent when $|x|+|y|<1;~a,b,c,d,g\in\mathbb{C}\backslash\mathbb{Z}^-_0$.

\item[{(ii)}] Appell's function $F_2$ is absolutely convergent when $|x|+|y|=1;\\~x\neq 0,~y\neq 0;~a,b,c,d,g\in\mathbb{C}\backslash\mathbb{Z}^-_0$ and $\mathfrak{R}(a+b+c-d-g)<0$.

\item[{(iii)}] When $a$ is a negative integer, then Appell's series $F_2$ will be a polynomial, $b,c,d,g\in\mathbb{C}\backslash\mathbb{Z}^-_0$.

\item[{(iv)}] When $b,c$ are negative integers, then Appell's series $F_2$ will be a polynomial, $a,d,g\in\mathbb{C}\backslash\mathbb{Z}^-_0$.
\end{enumerate}
For absolutely and conditionally convergence of Appell's double series $F_2 $, we refer a beautiful paper of H\`{a}i {\em et al.}\cite{Haisrimar92115}.\\

\noindent
Mellin-Barnes type contour integral representation of binomial function$~{}_1F_0(z)$:
\begin{equation}\label{D12.100}
(1-z)^{-a}=~{}_1F_0\left[\begin{array}{ll}
~a;\\
&z\\
-;\end{array}\right]=\frac{1}{2\pi i~\Gamma(a)}\int_{-i\infty}^{+i\infty}\Gamma(a+s)\Gamma(-s)(-z)^s~ds:~~z\neq 0,
\end{equation}
where $|\arg(-z)|<\pi, 
a\in\mathbb{C}\backslash\mathbb{Z}^-_0$ and $i=\sqrt{(-1)}$.\\

\noindent
Mellin-Barnes type contour integral representation of $~{}_pF_q(z)$ \cite[p.43, Eq.(6)]{SriMano}:
\begin{equation}\label{D12.101}
~{}_pF_q\left[\begin{array}{ll}
\alpha_1,\alpha_2,...,\alpha_p;\\
& z\\
\beta_1,\beta_2,...,\beta_q;\end{array}\right]=\frac{\Gamma(\beta_1)\Gamma(\beta_2)...\Gamma(\beta_q)}{\Gamma(\alpha_1)\Gamma(\alpha_2)...\Gamma(\alpha_p)}.\frac{1}{2\pi i}\int_{\sigma-i\infty}^{\sigma+i\infty}\frac{\Gamma(\alpha_1+\xi)...\Gamma(\alpha_p+\xi)\Gamma(-\xi)(-z)^\xi}{\Gamma(\beta_1+\xi)...\Gamma(\beta_q+\xi)}~d\xi,
\end{equation}
where $z\neq 0$.\\
\noindent
 Convergence conditions:\\
If $p=q+1$, then $|\arg(-z)|<\pi$.\\
If $p=q$, then $|\arg(-z)|<\frac{\pi}{2}$,\\
and no $\alpha_i~(i=1,2,...,p)$ 
is zero or negative integer but some of $\beta_j~(j=1,2,...,q)$ may be zero or negative integers.\\

\noindent
Mellin-Barnes type contour integral representation of Meijer's G-function (\cite[p.45, Eq.(1)]{SriMano}, see also \cite{Erd1,Luke169}):\\
When $p\leq q$ and $1\leq m \leq q,~0\leq n\leq p,$ then
\begin{equation*}
G^{m,n}_{p,q}\left(z\left|   \begin{array}{ll}
\alpha_1,\alpha_2,\alpha_3,...,\alpha_n;\alpha_{n+1},...,\alpha_p\\
\beta_1,\beta_2,\beta_3,...,\beta_m;\beta_{m+1},...,\beta_q  \end{array}\right. \right)=~~~~~~~~~~~~~~~~~~~~~~~~~~~~~~~~~~~
\end{equation*}
\begin{equation}\label{E12.12}
=\frac{1}{2\pi i}\int_{-i\infty}^{+i\infty}\frac{\Gamma(\beta_1-s)...\Gamma(\beta_m-s)\Gamma(1-\alpha_1+s)...\Gamma(1-\alpha_n+s)}{\Gamma(1-\beta_{m+1}+s)...\Gamma(1-\beta_q+s)\Gamma(\alpha_{n+1}-s)...\Gamma(\alpha_p-s)}(z)^s~ds,
\end{equation}
where $z\neq 0,(\alpha_i-\beta_j)\neq$ positive integers, $i=1,2,3,...,n;~j=1,2,3,...,m$. For details of contours, see\cite[p.207, \cite{Luke169}, p.144]{Erd1}.\\

\noindent
Relations between Meijer's $G$- function and ${}_2F_1(z)$ \cite[p.61, \cite{Wikipedia}, p.77, Eq.(1)]{Mathaisaxena}: 
\begin{equation*}
G^{2~2}_{2~2}\left(z\left|   \begin{array}{ll}
1-a,1-b;-\\
0,c-a-b;-  \end{array}\right. \right)=\frac{\Gamma(a)\Gamma(b)\Gamma(c-a)\Gamma(c-b)}{\Gamma(c)}{}_2F_1\left[\begin{array}{ll}
a,~b;\\
& 1-z\\
~~~ ~c;\end{array}\right],
\end{equation*}
where $|1-z|<1$ and $c-a,c-b\neq 0,-1,-2,...$
\begin{equation*}
G^{2~2}_{2~2}\left(z\left|   \begin{array}{ll}
a_1,a_2;-\\
b_1,b_2;-  \end{array}\right. \right)=\frac{\Gamma(1-a_1+b_1)\Gamma(1-a_1+b_2)\Gamma(1-a_2+b_1)\Gamma(1-a_2+b_2)z^{b_1}}{\Gamma(2-a_1-a_2+b_1+b_2)}\times
\end{equation*}
\begin{equation}\label{F12.200}
\times{}_2F_1\left[\begin{array}{ll}
1-a_1+b_1,1-a_2+b_1;\\
& 1-z\\
~~~2-a_1-a_2+b_1+b_2;\end{array}\right];~~~|1-z|<1.
\end{equation}

\noindent
A unification of Lauricella's fourteen triple hypergeometric functions $F_1,F_2,...,F_{14}$\cite{Lauricella93158}, including ten functions in Saran's notations  $F_E,F_F,F_G,F_K,F_M,F_N,F_P,F_R,F_S$ and $F_T$\cite{Saran1954,Saran1956} and the three additional triple hypergeometric functions $H_A,H_B,H_C$\cite{Srivastava1964,Srivas1967} was introduced by Srivastava \cite{Srivastava1967} who defined a general triple hypergeometric series $F^{(3)}[x,y,z]$ (\cite[p.428]{Srivastava1967}, see also \cite[p.156, \cite{Deshpande1981}, p.40]{ChhabRusia1980}):
\begin{equation*}
F^{(3)}[x,y,z]=F^{(3)}\left[\begin{array}{ll}
(a_A)\text{\bf::}~(b_B);(b'_{B'});~(b''_{B''})\text{\bf:}~(c_C);(c'_{C'});(c''_{C''})\text{\bf;}\\
&x,~y,~z\\
(e_E)\text{\bf::}~(g_G);(g'_{G'});(g''_{G''})\text{\bf:}(h_H);(h'_{H'});(h''_{H''})\text{\bf;}\end{array}\right]
\end{equation*}
\begin{equation}\label{E12.20}
=\sum_{m,n,p=0}^{\infty}\Lambda(m,n,p)~\frac{x^m}{m!}~\frac{y^n}{n!}~\frac{z^p}{p!},
\end{equation}
where, for convenience,
\begin{equation*}
\Lambda(m,n,p)=\frac{\prod_{j=1}^{A}(a_j)_{m+n+p}~\prod_{j=1}^{B}(b_j)_{m+n}~\prod_{j=1}^{B'}(b'_j)_{n+p}~\prod_{j=1}^{B''}(b''_j)_{m+p}}{\prod_{j=1}^{E}(e_j)_{m+n+p}~\prod_{j=1}^{G}(g_j)_{m+n}~\prod_{j=1}^{G'}(g'_j)_{n+p}~\prod_{j=1}^{G''}(g''_j)_{m+p}}\times
\end{equation*}
\begin{equation}\label{A13.2}
\times\frac{\prod_{j=1}^{C}(c_j)_{m}~\prod_{j=1}^{C'}(c'_j)_{n}~\prod_{j=1}^{C''}(c''_j)_{p}}{\prod_{j=1}^{H}(h_j)_{m}~\prod_{j=1}^{H'}(h'_j)_{n}~\prod_{j=1}^{H''}(h''_j)_{p}}
\end{equation}
where $(a_A)$ abbreviates the array of $A$ parameters $a_1,a_2,...,a_A$, with similar interpretations for $(b_B),(b'_{B'}),(b''_{B''}),  et~cetera$. The triple hypergeometric series in (\ref{E12.20}) converges when 

$$
\left\{\begin{array}{llrrr}
1+E+G+G''+H-A-B-B''-C\geq 0,\\
1+E+G+G'+H'-A-B-B'-C'\geq 0,\\
1+E+G'+G''+H''-A-B'-B''-C''\geq 0,\\
\end{array}\right.\\$$
where the equalities hold true for suitably constrained values of $|x|,|y|$ and $|z|$.\\

\noindent
Suppose $\phi(x,y)=0$ is the projection of the curved surface of three dimensional figure $z=f(x,y)$ over the $x$-$y$ plane, then curved surface area is given by
\begin{equation}\label{a12.9}
\hat{S}=\underbrace{\int\int}_{\stackrel{\text{over the area }} {\phi(x,y)=0}}\sqrt{\left\lbrace 1+\left( \frac{\partial z}{\partial x}\right) ^2+\left( \frac{\partial z}{\partial y}\right) ^2\right\rbrace }~dx~dy.
\end{equation}
\begin{figure}[ht!]
\noindent\begin{minipage}{0.45\textwidth}
\includegraphics[width=\linewidth]{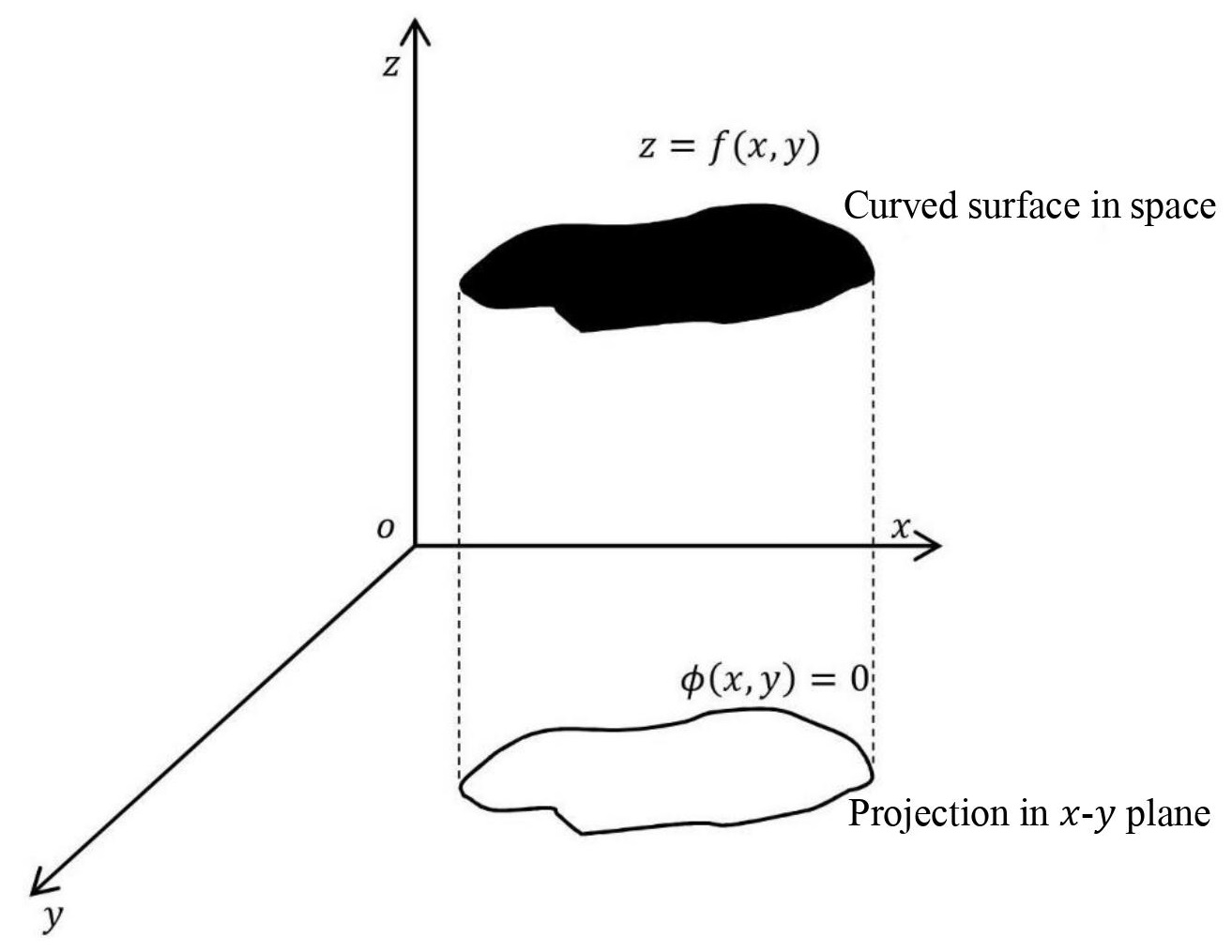}
\end{minipage}%
\caption[short]{Projection of curved surface in $x$-$y$ plane.}
\label{fig:Projection }
\end{figure}

\begin{equation}\label{D12.1000}
\text{A definite integral }\int_{\theta=0}^{\frac{\pi}{2}}\sin^\alpha {\theta}\cos^\beta {\theta}~d\theta=\frac{\Gamma\left(\frac{\alpha+1}{2}\right) \Gamma\left(\frac{\beta+1}{2}\right) }{2\Gamma\left(\frac{\alpha+\beta+2}{2}\right) };~\mathfrak{R}(\alpha)>-1,~\mathfrak{R}(\beta)>-1.
\end{equation}

Special functions arise in the solution of various classical problems of physics, generally involving the flow of electromagnetic, acoustic or thermal energy. In the study of propagation of heat in a metallic bar, one could consider a bar with a rectangular cross-section, a round cross-section, an elliptic cross-section or even more complicated cross-sections. In these situations while dealing with this kind of problem, leads to ordinary differential equations whose solution forms the majority of special functions of mathematical physics. In order to determine the allowable frequencies of oscillation of the waves on the surface of fluid put in a circular dish, hypergeometric series are required to terminate and determine the allowable frequencies. The hypergeometric functions have been turned out to have a variety of applications in a wide range of research subjects (for example see \cite{Akram2022,Akkurt2021}). In this paper we derive the closed form for obtaining the exact curved surface area of a frustum of hemiellipsoid by using Mellin-Barnes type contour integral representations of generalized hypergeometric function$~_pF_q(z)$ in terms of Appell's function of second kind and general triple hypergeometric series of Srivastava in Section \ref{T20.3} with the aid of some important definite integrals $\int_{\theta=-\pi}^{\pi}\left(\frac{\cos^2{\theta}}{\beta^2}+\frac{\sin^2{\theta}}{\lambda^2}\right)^sd\theta$ and $\int_{r=\beta}^{\gamma} \frac{r^{2s+1}}{(1-r^2)^s}dr$ given in  Section \ref{T20.2}. The problem of approximating the surface area of
a triaxial hemiellipsoid does not appear to have been addressed i.e, no closed form expression exists for the surface area of a hemiellipsoid. This situation stems from the fact that it is impossible to execute the integration in the expression for the
surface area in closed form for the most general case of three unequal axes. Some numerical examples have also been demonstrated in 
Section \ref{HE1} among numerous ones.
\section{Evaluation of some useful definite integrals
}\label{T20.2}
\noindent
The following definite integrals hold true associated with suitable convergence conditions:
\begin{theorem}\label{TH1}
\begin{equation}\label{G12.10}
\int_{\theta=-\pi}^{\pi}\left(\frac{\cos^2{\theta}}{\sigma^2}+\frac{\sin^2{\theta}}{\lambda^2}\right)^s~d\theta =\frac{2\pi \lambda}{\sigma^{1+2s}}~{}_2F_1\left[\begin{array}{ll}
\frac{1}{2},1+s;\\
&1-\frac{\lambda^2}{\sigma^2}\\
~~~~~~~~1;\end{array}\right],~~~~~~~~~~~~~~~
\end{equation}
where $\sigma\geq \lambda>0$ and it is obvious that $0\leq(1-\frac{\lambda^2}{\sigma^2})<1$.\\
\end{theorem}
\begin{theorem}\label{TH2}
\begin{equation}\label{G12.10000}
\int_{\theta=-\pi}^{\pi}\left(\frac{\cos^2{\theta}}{\sigma^2}+\frac{\sin^2{\theta}}{\lambda^2}\right)^s~d\theta =\frac{2\pi \sigma}{\lambda^{1+2s}}~{}_2F_1\left[\begin{array}{ll}
\frac{1}{2},1+s;\\
&1-\frac{\sigma^2}{\lambda^2}\\
~~~~~~~~1;\end{array}\right],~~~~~~~~~~~~~~~
\end{equation}
where $\lambda\geq \sigma>0$ and it is obvious that $0\leq(1-\frac{\sigma^2}{\lambda^2})<1$.\\
\end{theorem}

\begin{proof}
For the independent demonstration of the Theorem \ref{TH1}
\begin{equation*}
\text{suppose} ~I_1=\int_{\theta=-\pi}^{\pi}\left(\frac{\cos^2{\theta}}{\sigma^2}+\frac{\sin^2{\theta}}{\lambda^2}\right)^s~d\theta.
\end{equation*}

\begin{equation*}
=\frac{4}{\lambda^{2s}}\int_{\theta=0}^{\frac{\pi}{2}}\left( \sin^2 \theta\right)^s\left\lbrace 1+\frac{\lambda^2\cos^2{\theta}}{\sigma^2\sin^2{\theta}}\right\rbrace ^s~d\theta 
\end{equation*}
\begin{equation}\label{G12.100}
=\frac{4}{\lambda^{2s}}\int_{\theta=0}^{\frac{\pi}{2}} \sin^{2s} \theta~{}_1F_0\left[\begin{array}{ll}
~~-s;\\
&\frac{-\lambda^2\cos^2{\theta}}{\sigma^2\sin^2{\theta}}\\
~-~~;\end{array}\right]~d\theta. 
\end{equation}
Employing the contour integral (\ref{D12.100}) of ${}_{1}F_{0}(.)$, we get

\begin{equation}\label{G12.1005}
I_1=\frac{2}{\pi i \Gamma(-s)~\lambda^{2s}}\int_{\theta=0}^{\frac{\pi}{2}} \sin^{2s}\theta\left\lbrace \int_{\zeta=-i\infty}^{+i\infty}\Gamma(-\zeta)\Gamma(-s+\zeta)\left(\frac{\lambda^2\cos^2{\theta}}{\sigma^2\sin^2{\theta}}\right)^\zeta~d\zeta \right\rbrace  ~d\theta.
\end{equation}
Upon interchanging the order of integration in double integral of (\ref{G12.1005}), and using the formula (\ref{D12.1000}), we get

\begin{equation}
I_1=\frac{1}{\pi i \Gamma(-s)\Gamma(1+s)~\lambda^{2s}}\int_{\zeta=-i\infty}^{+i\infty}\Gamma(0-\zeta)\Gamma(1-(s+1)+\zeta)\Gamma\left(\frac{1}{2}+s-\zeta\right)\Gamma\left(1-\frac{1}{2}+\zeta\right)\left(\frac{\lambda^2}{\sigma^2}\right)^\zeta~d\zeta.
\end{equation}
Applying Meijer's $G$-function (\ref{E12.12}), we get
\begin{equation}\label{G12.101}
I_1=\frac{2}{\Gamma(-s)\Gamma(1+s)~\lambda^{2s} }~G^{2~2}_{2~2}\left(\frac{\lambda^2}{\sigma^2}\left|   \begin{array}{ll}
s+1,\frac{1}{2};-\\
0,\frac{1}{2}+s;-  \end{array}\right. \right).~~~~~~~~~~~~~~~~~~~~~~~~~~~~~~~~~~~~
\end{equation}
Employing the conversion formula (\ref{F12.200}) in equation (\ref{G12.101}), and after further simplification, we arrive at the result asserted in Theorem \ref{TH1}.\\

\noindent
The proof of the Theorem \ref{TH2} follows the same steps as in the proof of the Theorem \ref{TH1}. So we omit the details here.
\end{proof}
\begin{theorem}\label{TH3}
\begin{equation}\label{p12.1000}
\int_{r=\beta}^{\gamma} \frac{r^{2s+1}}{(1-r^2)^s}~~dr=\sum_{\gamma\longrightarrow\beta}^{\circledast}\left\lbrace \frac{\gamma^{2+2s}} {2(1+s)}~{}_2F_1\left[\begin{array}{ll}
s,1+s;\\
& \gamma^2\\
~~~2+s;\end{array}\right]\right\rbrace,~~~~~~~~~~~~~~~~~~~~~
\end{equation}
where $~0<\beta<\gamma<1$. 
\end{theorem}
\begin{proof}
When $0<\beta<r<\gamma<1$, then
\begin{equation*}
\int_{r=\beta}^{r=\gamma} \frac{r^{2s+1}}{(1-r^2)^s}~~dr=\int_{\beta}^{\gamma}r^{2s+1}{(1-r^2)^{-s}}~~dr~~~~~~~~~~~~~~~~~~~~~~~~~~~~~~~~~~~~~~~~~~~`
\end{equation*}

\begin{equation}
=\int_{\beta}^{\gamma} r^{2s+1}~{}_1F_0\left[\begin{array}{ll}
~s~;\\
& r^2\\
~-;\end{array}\right]~dr~~~~~~~~~~~~
\end{equation}

\begin{equation}
=\sum_{p=0}^{\infty}\frac{(s)_p~}{p!} \int_{\beta}^{\gamma}~r^{2p+2s+1}~dr ~~~~~~~~~~~~~~~~~
\end{equation}

\begin{equation}
=\frac{1}{2(1+s)}~\sum_{p=0}^{\infty}\frac{(s)_p~(1+s)_p~}{(2+s)_p ~p!} ~\left( \gamma^{2p+2s+2}-\beta^{2p+2s+2}\right). 
\end{equation}
After further simplification, we arrive at the result as asserted in Theorem \ref{TH3}.\\
\end{proof}
\noindent
Throughout the paper, the semi-axes of the ellipsoid $\frac{x^2}{a^2}+\frac{y^2}{b^2}+\frac{z^2}{c^2}=1$ are assumed in the form of $a>b>c>0$.
\section{Closed form for the curved surface area of a frustum of hemiellipsoid }\label{T20.3}

\noindent
For the sake of convenience, we shall use the following notation to represent the lengthy mathematical expression
\begin{equation}
\sum_{\gamma\longrightarrow\beta}^{\circledast}\left\lbrace\Phi(\gamma) \right\rbrace =\left\lbrace \Phi(\gamma)\right\rbrace -\left\lbrace\Phi(\beta)\right\rbrace .~~~~~~~~~~~~~~~~~~~~~~~~~~~~~~~~~~~~~~
\end{equation}
\noindent
\begin{theorem}\label{TH4}
 The curved surface area of a frustum of hemiellipsoid (whose axis is positive direction of $z$-axis) i.e, $z=c\left( 1-\frac{x^2}{a^2}-\frac{y^2}{b^2}\right)^{\frac{1}{2}};a>b>c>0$, lying above $x$-$y$ plane (i.e, $z=0$) and bounded by the two parallel planes $z=h$ and $z=H~$ is given by:
\begin{equation*}
\hat{S}~=\sum_{\gamma\longrightarrow\beta}^{\circledast}\left\lbrace b^2\gamma^2 \pi~F_2\left[1;~\frac{1}{2},\frac{-1}{2}~;~1,2~;1-\frac{b^2}{a^2},\frac{-c^2\gamma^2}{a^2} \right]+ \right.~~~~~~~~~~~~~~~~~~~~~~~~
\end{equation*}
\begin{equation}\label{H12.10}
\left.+ \frac{b^2c^2\gamma^6 \pi}{6a^2}~F^{(3)}\left[\begin{array}{ll}
-~\text{\bf {:}}\text{\bf {:}}~-;2,3;~2~\text{\bf {:}}~\frac{1}{2}~;~1~;~~~\frac{1}{2}~\text{\bf;}\\
& 1-\frac{b^2}{a^2},\gamma^2,\frac{-c^2\gamma^2}{a^2}\\
-~\text{\bf {:}}\text{\bf {:}}~-;~~4~;-~\text{\bf {:}}~1~;~2~;~2,2~\text{\bf;}\\
\end{array}\right] \right\rbrace ,  
\end{equation}
where $0<h<H<c;~\beta^2={\left(1- \frac{H^2}{c^2}\right) };~\gamma^2={\left(1- \frac{h^2}{c^2}\right) };~0<\beta<\gamma<1;~|1-\frac{b^2}{a^2}|<1;~a>b>c>0;~|\frac{-c^2\beta^2}{a^2}|<1$ and $|\frac{-c^2\gamma^2}{a^2}|<1$.\\
\end{theorem}
\noindent
{\bf Remark:} The above formula (\ref{H12.10}) is verified numerically through {\em Mathematica program}.\\
\begin{figure}[ht!]
\noindent\begin{minipage}{0.50\textwidth}
\includegraphics[width=\linewidth]{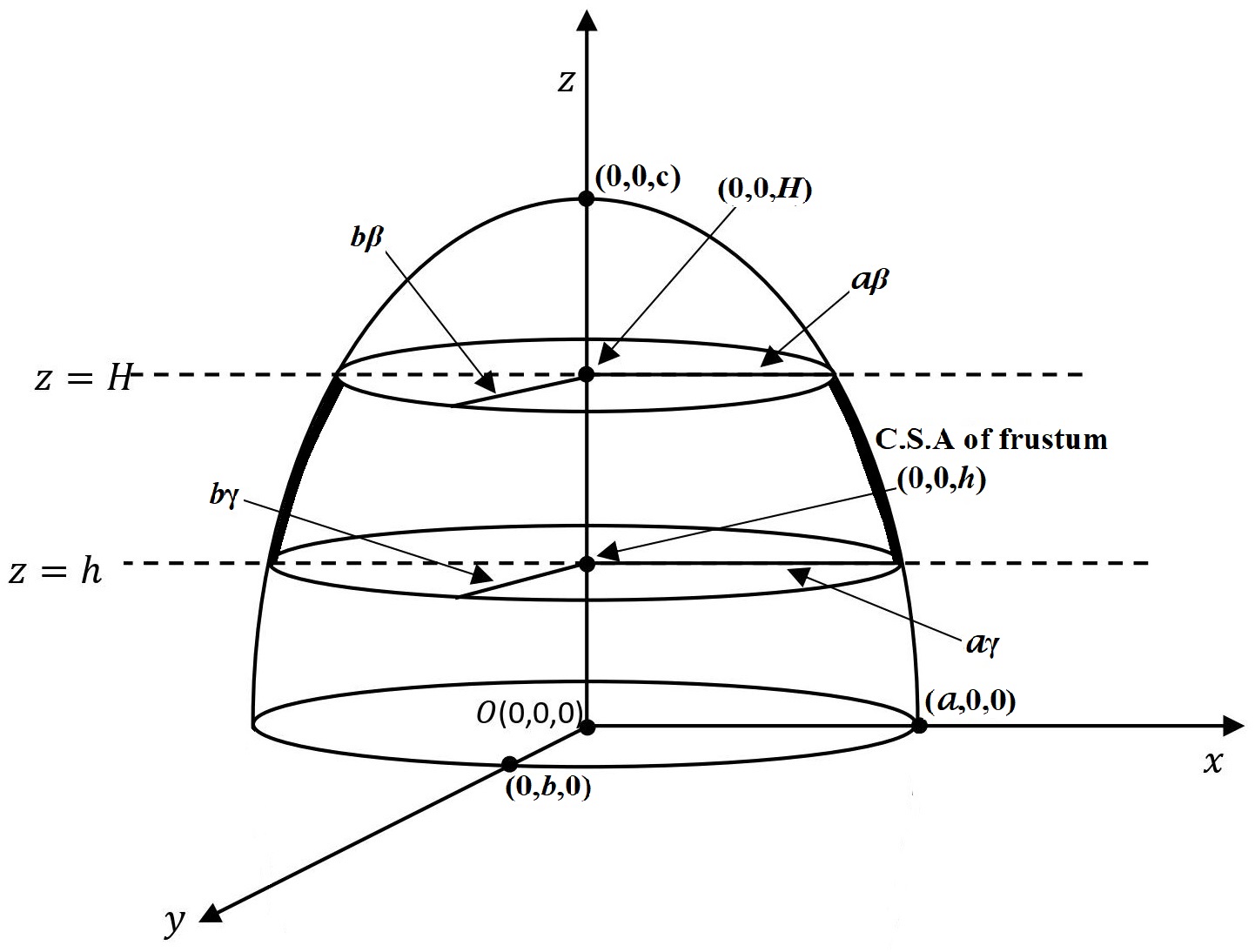}
\end{minipage}%
\caption[short]{Hemiellipsoid $z=c\left( 1-\frac{x^2}{a^2}-\frac{y^2}{b^2}\right)^{\frac{1}{2}}$ lying above the $x$-$y $ plane}
\label{fig:ellipsoid}
\end{figure}

\begin{proof}
The equation of an ellipsoid is given by
\begin{equation}\label{p12.1}
\frac{x^2}{a^2}+\frac{y^2}{b^2}+\frac{z^2}{c^2}=1;~~a>b>c>0.~~~~~~~~~~~~~~~
\end{equation}
Intersection of the surface (\ref{p12.1}) with the planes $z=H$ and $z=h$, where $c>H>h>0$ (parallel to $x$-$y$ plane, lying above $x$-$y$ plane)

\begin{equation*}
\frac{x^2}{a^2}+\frac{y^2}{b^2}=1-\frac{H^2}{c^2}=\beta^2<1~~~~~~~~~~~~~~~~~~~~~~~~~~~~~~~~~~~~~~~~~~~~~~~~~~~~~
\end{equation*}
\begin{equation}
\text{and}~~~~~\frac{x^2}{a^2}+\frac{y^2}{b^2}=1-\frac{h^2}{c^2}=\gamma^2<1,~~~~~~~~~~~~~~~~~~~~~~~~~~~~~~~~~~~~~~~~~~~~~~~~~~~~~
\end{equation}
where $\gamma=\sqrt{\left(1-\frac{h^2}{c^2}\right) },~\beta=\sqrt{\left(1-\frac{H^2}{c^2}\right) }$ and $0 <\beta<\gamma<1$.\\

\noindent
Again from equation (\ref{p12.1}), we have
\begin{equation}
\frac{\partial z}{\partial x}=\frac{-cx}{a^2\sqrt{\left( 1-\frac{x^2}{a^2}-\frac{y^2}{b^2}\right)}},~\frac{\partial z}{\partial y}=\frac{-cy}{b^2\sqrt{\left( 1-\frac{x^2}{a^2}-\frac{y^2}{b^2}\right)}}. ~~~~~~~~~~~~~~~~~~~~~~~~~~~~~~~~~~~~~~~~~~~~~~~~~
\end{equation}

The projection of curved surface area of a frustum of hemiellipsoid in $x $-$y$ plane will be the area between the two concentric ellipses, given by.\\

\begin{equation}
\frac{x^2}{(a\gamma)^2}+\frac{y^2}{(b\gamma)^2}=1~~\text{and}~~\frac{x^2}{(a\beta)^2}+\frac{y^2}{(b\beta)^2}=1;~~a>b,~\gamma>\beta
\end{equation}

\begin{figure}[ht!]
\noindent\begin{minipage}{0.50\textwidth}
\includegraphics[width=\linewidth]{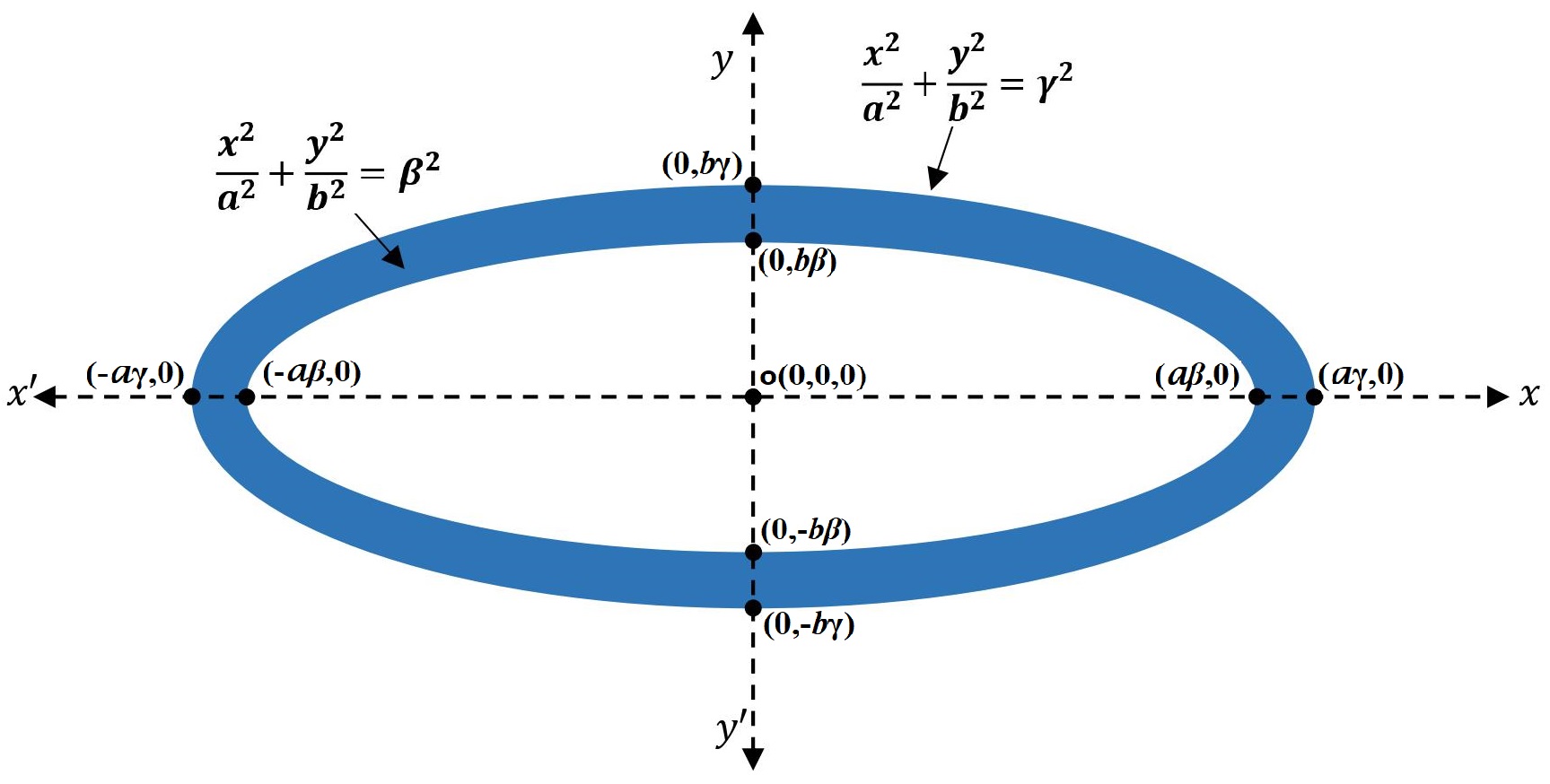}
\end{minipage}%
\caption[short]{Two concentric ellipses $\frac{x^2}{a^2}+\frac{y^2}{b^2}=\gamma^2,~\frac{x^2}{a^2}+\frac{y^2}{b^2}=\beta^2$.}
\label{fig:ellipes}
\end{figure}
\noindent
Substitute the values of $\frac{\partial z}{\partial x}$ and $\frac{\partial z}{\partial y}$ in equation (\ref{a12.9}).
Therefore curved surface area of a frustum of hemiellipsoid will be
\begin{equation}
\hat{S}=\underbrace{\int\int}_{\stackrel{\text{over the area between}}{\text{ two concentric ellipses}}}\sqrt{\left\lbrace 1+\frac{c^2x^2}{a^4\left( 1-\frac{x^2}{a^2}-\frac{y^2}{b^2}\right)}+\frac{c^2y^2}{b^4\left( 1-\frac{x^2}{a^2}-\frac{y^2}{b^2}\right)}\right\rbrace }~dx ~dy.
\end{equation}
Put $x=aX,~y=bY$, therefore, 
\begin{equation}
\hat{S}=ab~\underbrace{\int\int}_{\stackrel{\text{over the area between two concentric}}{\text{  circles}X^2+Y^2=\gamma^2,X^2+Y^2=\beta^2}}\sqrt{\left\lbrace 1+\frac{c^2X^2}{a^2\left( 1-X^2-Y^2\right)}+\frac{c^2Y^2}{b^2\left( 1-X^2-Y^2\right)}\right\rbrace }~dX ~dY.
\end{equation}

\begin{figure}[ht!]
\noindent\begin{minipage}{0.50\textwidth}
\includegraphics[width=\linewidth]{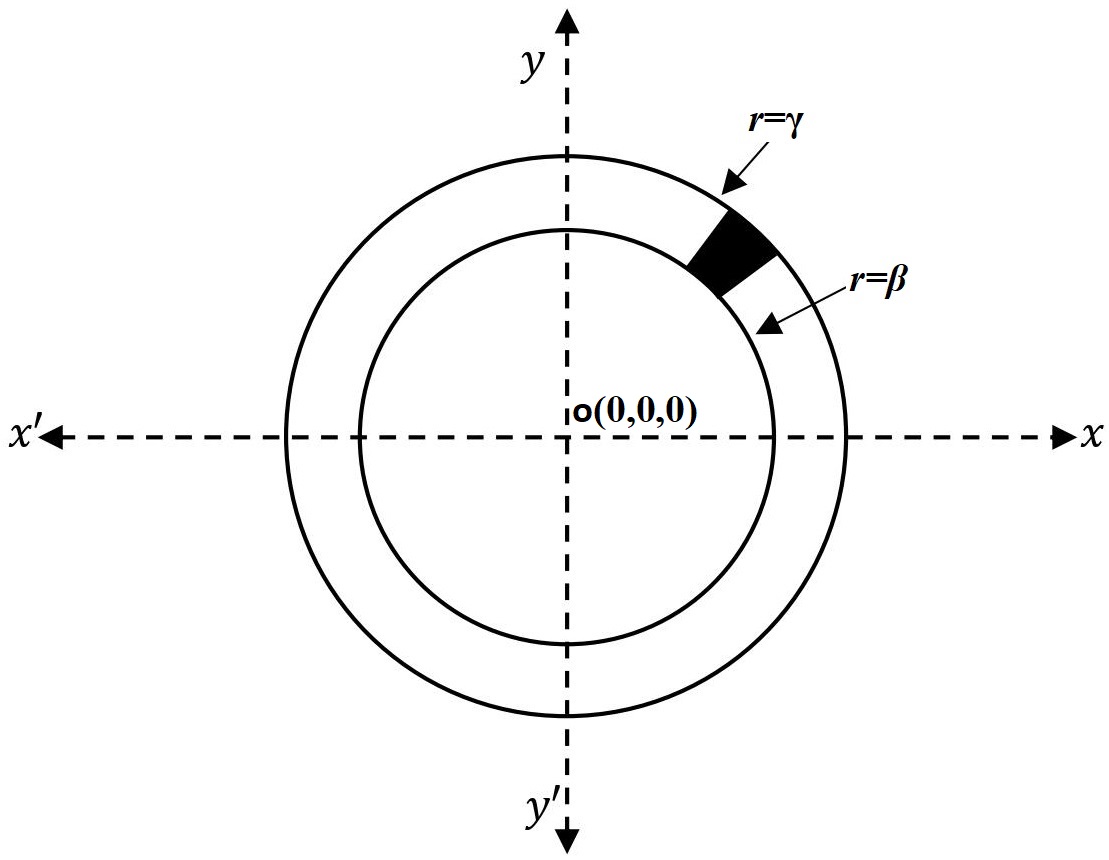}
\end{minipage}%
\caption[short]{Concentric circles $r=\gamma,~r=\beta$ in polar form }
\label{fig:circles}
\end{figure}

Put $X=r\cos \theta,~Y=r\sin \theta$, then $dX~dY=rdr~d\theta$
\begin{equation}\label{13x.120}
\text{Therefore}~\hat{S}=ab~\int_{\theta=-\pi}^{\pi}\left( \int_{r=\beta}^{\gamma}\left\lbrace 1+\frac{c^2r^2}{\left( 1-r^2\right)}\left( \frac{\cos^2\theta}{a^2}+\frac{\sin^2\theta}{b^2}\right) \right\rbrace^{\frac{1}{2}} ~rdr\right)  ~d\theta
\end{equation}
Since due to the non availability of a standard formula of definite/indefinite integrals in the literature of integral calculus for the integration with respect to $"r"$ and $"\theta"$ in double integral (\ref{13x.120}). We solve such integrals exactly through hypergeometric function approach.
\begin{equation}\label{13B.120}
\text{Therefore}~\hat{S}=ab~\int_{\theta=-\pi}^{\pi}\int_{r=\beta}^{\gamma}{}_{1}F_{0}\left[\begin{array}{ll}
\frac{-1}{2};\\
& -\frac{c^2r^2}{\left( 1-r^2\right)}\left( \frac{\cos^2\theta}{a^2}+\frac{\sin^2\theta}{b^2}\right) \\
~-;\end{array}\right]~rdr ~d\theta
\end{equation}
Due to uncertainty in the argument of $~_1F_0(.)$ in equation (\ref{13B.120}), we apply contour integral (\ref{D12.100}) of ${}_{1}F_{0}(.)$ in equation (\ref{13B.120}), we get

\begin{equation}\label{L13.9999}
\hat{S}=ab~\int_{\theta=-\pi}^{\pi}\int_{r=\beta}^{\gamma}\left[ \frac{1}{2\pi i \Gamma\left(\frac{-1}{2}\right) }~{\int}_{s=-i\infty}^{+i\infty}~\Gamma(-s)\Gamma( \frac{-1}{2}+s) \left\lbrace \frac{c^2r^2}{(1-r^2)}\left( \frac{\cos^2\theta}{a^2}+\frac{\sin^2\theta}{b^2}\right) \right\rbrace^s ds \right]  ~rdr~d\theta,
\end{equation}
where $|\arg\left\lbrace c^2\left( \frac{\cos^2\theta}{a^2}+\frac{\sin^2\theta}{b^2}\right) \right\rbrace|<\pi$.\\

\noindent
Interchanging the order of integration in double integral of (\ref{L13.9999})and employing the formulas (\ref{G12.10}) and (\ref{p12.1000}), we get

\begin{equation*}
\hat{S}=\sum_{\gamma\longrightarrow\beta}^{\circledast}\left\lbrace \frac{-b^2\gamma^2}{4 \sqrt{\pi}~i}~{\int}_{s=-i\infty}^{+i\infty}~
 \frac{\Gamma(-s)\Gamma( \frac{-1}{2}+s)}{(1+s)}\left( \frac{c\gamma}{a}\right)^{2s}~{}_{2}F_{1}\left[\begin{array}{ll}
\frac{1}{2},1+s;\\
& 1-\frac{b^2}{a^2}\\
~~~~~~~~1;\end{array}\right]\times\right.
\end{equation*}
\begin{equation}\label{H12.17}
\left.\times~{}_{2}F_{1}\left[\begin{array}{ll}
s,1+s;\\
&\gamma^2\\
~~~2+s;\end{array}\right]ds\right\rbrace , 
\end{equation}
where $|\gamma^2|<1$ and $~|1-\frac{b^2}{a^2}|<1$.

\begin{equation*}
\hat{S}=\sum_{\gamma\longrightarrow\beta}^{\circledast}\left\lbrace \frac{-b^2\gamma^2}{4 \sqrt{\pi}~i}~{\int}_{s=-i\infty}^{+i\infty}~
 \frac{\Gamma(-s)\Gamma( \frac{-1}{2}+s)}{(1+s)}\left( \frac{c\gamma}{a}\right)^{2s}~\left( \sum_{m=0}^{\infty}\frac{
\left( \frac{1}{2}\right)_m(1+s)_m\left( 1-\frac{b^2}{a^2}\right) ^m}{(1)_m~m!}\right) ~ds-\frac{b^2\gamma^2}{4 \sqrt{\pi}~i}\times\right.
\end{equation*}
\begin{equation}\label{P12.17}
\left.\times{\int}_{s=-i\infty}^{+i\infty}~
 \frac{\Gamma(-s)\Gamma( \frac{-1}{2}+s)}{(1+s)}\left( \frac{c\gamma}{a}\right)^{2s}~\left( \sum_{m=0}^{\infty}\frac{
\left( \frac{1}{2}\right)_m(1+s)_m\left( 1-\frac{b^2}{a^2}\right) ^m}{(1)_m~m!}\right)\left(  \sum_{n=1}^{\infty}\frac{
(s)_n(1+s)_n~\gamma ^{2n}}{(2+s)_n~n!}\right) ~ds\ \right\rbrace .
\end{equation}

\noindent
On interchanging the order of summation and integration, we get
\begin{equation*}
\hat{S}=\sum_{\gamma\longrightarrow\beta}^{\circledast}\left\lbrace \frac{-b^2\gamma^2}{4 \sqrt{\pi}~i}~\sum_{m=0}^{\infty}\frac{
\left( \frac{1}{2}\right)_m\left( 1-\frac{b^2}{a^2}\right) ^m}{(1)_m~m!}{\int}_{s=-i\infty}^{+i\infty}~
 \frac{\Gamma( \frac{-1}{2}+s)\Gamma( 1+m+s)\Gamma(-s)}{\Gamma(2+s)}\left( \frac{c^2\gamma^2}{a^2}\right)^{s}~ds-\right.~~~~~~~~~~~~~~~~~~~~~~~~~~~~~~~~~~~~~~~~~~
\end{equation*}
\begin{equation*}
-\frac{b^2\gamma^2}{4 \sqrt{\pi}~i}\sum_{m=0}^{\infty}\sum_{n=1}^{\infty}\frac{
\left( \frac{1}{2}\right)_m\left( 1-\frac{b^2}{a^2}\right) ^m\gamma ^{2n}}{(1)_m~m!~n!}\times~~~~~~~~~~~~~~~~~~~
\end{equation*}
\begin{equation}\label{p12.27}
\left.\times{\int}_{s=-i\infty}^{+i\infty}~
 \frac{\Gamma( \frac{-1}{2}+s)\Gamma( 1+m+s)\Gamma(n+s)\Gamma( 1+n+s)\Gamma(-s)}{\Gamma(0+s)\Gamma( 1+s)\Gamma( 2+n+s)}\left( \frac{c^2\gamma^2}{a^2}\right)^{s}~ds\right\rbrace .
\end{equation}
Replacing $n $ by $n+1$ in equation (\ref{p12.27}) and applying contour integral (\ref{D12.101}), we get
\begin{equation*}
\hat{S}=\sum_{\gamma\longrightarrow\beta}^{\circledast}\left\lbrace b^2\gamma^2\pi\sum_{m=0}^{\infty}\frac{
\left( \frac{1}{2}\right)_m\left( 1-\frac{b^2}{a^2}\right) ^m}{(1)_m}~_2{F}_1\left[\begin{array}{lll}
\frac{-1}{2},1+m;\\
& \frac{-c^2\gamma^2}{a^2}\\
~~~~~~~~~~2~;\\\end{array}\right]+\right. ~~~~~~~~~~~~~~~~~~~~~
\end{equation*}
\begin{equation}\label{A13.51}
\left.+b^2\gamma^4\pi\sum_{m=0}^{\infty}\sum_{n=0}^{\infty}\frac{
\left( \frac{1}{2}\right)_m\left( 1-\frac{b^2}{a^2}\right) ^m~\gamma^{2n}~\Gamma(1+n)}{(1)_m~\Gamma(3+n)~}~\left( \frac{1}{\Gamma(0)}~_4{F}_3\left[\begin{array}{lll}
\frac{-1}{2},1+m,1+n,2+n;\\
& \frac{-c^2\gamma^2}{a^2}\\
~~~~~~~~~~~~~0,~1,~3+n~;\\\end{array}\right]\right) \right\rbrace ,
\end{equation}
where $\frac{-c^2\gamma^2}{a^2}<1.$\\

On further simplification, we arrive at the result as asserted in Theorem \ref{TH4} with the aid of the assertion given in (\ref{A13.50}).
\end{proof}

\section{Applications of the closed form (\ref{H12.10})}\label{HE1}

\noindent
To understand the subject matter, we shall discuss some numerical problems:\\

 \noindent
 {\bf Problem 1.} Find the exact curved surface area of a frustum of hemiellipsoid $z=\left( 1-\frac{x^2}{9}-\frac{y^2}{4}\right)^{\frac{1}{2}}$ bounded by the planes $z=\sqrt{0.6}$ and $z=\sqrt{0.2}$.\\
 
 \noindent
 {\bf Solution: } We have $a=3,~b=2,~c=1,~H=\sqrt{0.6}=0.774596669...\text{ and}~h=\sqrt{0.2}=0.447213595...,$ then $\beta=0.2\text{ and }\gamma =0.6$\\
                      
Substituting the above values of $a,b,c,h,H,\beta$ and $\gamma$ in the closed form (\ref{H12.10}), we have
\begin{equation*}
\hat{S}=(1.44 \pi)F_2\left[1;\frac{1}{2},\frac{-1}{2};1,2;0.55...,-0.04\right]-(0.16\pi)F_2\left[1;\frac{1}{2},\frac{-1}{2};1,2;0.55...,-0.0044...\right]+
\end{equation*}
\begin{equation*}
+ (0.003456\pi)~F^{(3)}\left[\begin{array}{ll}
-~\text{\bf {:}}\text{\bf {:}}~-;2,3;~2~\text{\bf {:}}~\frac{1}{2}~;~1~;~~~\frac{1}{2}~\text{\bf;}\\
& 0.55...\text{\bf,}~0.36\text{\bf,}~0.04\\
-~\text{\bf {:}}\text{\bf {:}}~-;~~4~;-~\text{\bf {:}}~1~;~2~;~2,2~\text{\bf;}\\
\end{array}\right]-  
\end{equation*}
\begin{equation}\label{P12.50}
-(0.000004740...\pi)~F^{(3)}\left[\begin{array}{ll}
-~\text{\bf {:}}\text{\bf {:}}~-;2,3;~2~\text{\bf {:}}~\frac{1}{2}~;~1~;~~~\frac{1}{2}~\text{\bf;}\\
& 0.55...\text{\bf,}~0.04\text{\bf,}~0.0044...\\
-~\text{\bf {:}}\text{\bf {:}}~-;~~4~;-~\text{\bf {:}}~1~;~2~;~2,2~\text{\bf;}\\
\end{array}\right]. 
\end{equation}

Now using the Mathematica program to calculate the sum of the double and triple infinite series in equation (\ref{P12.50}), we get the {\bf exact} curved surface area of a frustum of hemiellipsoid $z=\left( 1-\frac{x^2}{9}-\frac{y^2}{4}\right)^{\frac{1}{2}}$ bounded by the planes $z=\sqrt{0.6}$ and $z=\sqrt{0.2}$ is $\hat{S}=6.1749238... ~\text{square units.}$\\

 \noindent
 {\bf Problem 2.} Find the exact curved surface area of a frustum of hemiellipsoid $z=2\left( 1-\frac{x^2}{25}-\frac{y^2}{9}\right)^{\frac{1}{2}}$ bounded by the planes $z=\sqrt{3.64}$ and $z=\sqrt{1.44}$.\\
 
 \noindent
 {\bf Solution: } We have $a=5,~b=3,~c=2,~H=\sqrt{3.64}=1.9078784...\text{ and}~h=\sqrt{1.44}=1.2,$ then $\beta=0.3\text{ and }\gamma=0.8$\\
                            
Substituting the above values of $a,b,c,h,H,\beta$ and $\gamma$ in the closed form (\ref{H12.10}), we have
\begin{equation*}
\hat{S}=(5.76 \pi)F_2\left[1;\frac{1}{2},\frac{-1}{2};1,2;0.64,-0.1024\right]-(0.81\pi)F_2\left[1;\frac{1}{2},\frac{-1}{2};1,2;0.64,-0.0144\right]+
\end{equation*}
\begin{equation*}
+ (0.0629146\pi)~F^{(3)}\left[\begin{array}{ll}
-~\text{\bf {:}}\text{\bf {:}}~-;2,3;~2~\text{\bf {:}}~\frac{1}{2}~;~1~;~~~\frac{1}{2}~\text{\bf;}\\
& 0.64\text{\bf,}~0.64\text{\bf,}~0.1024\\
-~\text{\bf {:}}\text{\bf {:}}~-;~~4~;-~\text{\bf {:}}~1~;~2~;~2,2~\text{\bf;}\\
\end{array}\right]-  
\end{equation*}
\begin{equation}\label{L12.50}
-(0.00017496\pi)~F^{(3)}\left[\begin{array}{ll}
-~\text{\bf {:}}\text{\bf {:}}~-;2,3;~2~\text{\bf {:}}~\frac{1}{2}~;~1~;~~~\frac{1}{2}~\text{\bf;}\\
& 0.64\text{\bf,}~0.09\text{\bf,}~0.0144\\
-~\text{\bf {:}}\text{\bf {:}}~-;~~4~;-~\text{\bf {:}}~1~;~2~;~2,2~\text{\bf;}\\
\end{array}\right]. 
\end{equation}

Now using the Mathematica program to calculate the sum of the double and triple infinite series in equation (\ref{L12.50}), we get the {\bf exact} curved surface area of a frustum of hemiellipsoid $z=2\left( 1-\frac{x^2}{25}-\frac{y^2}{9}\right)^{\frac{1}{2}}$ bounded by the planes $z=\sqrt{3.64}$ and $z=\sqrt{1.44}$ is $\hat{S}=28.41904... ~\text{square units.}$\\

\section{Conclusion}
In this paper, we obtained the closed form for the exact curved surface area of a frustum of hemiellipsoid $z=c\left( 1-\frac{x^2}{a^2}-\frac{y^2}{b^2}\right)^{\frac{1}{2}}$ obtained by the interception of a hemiellipsoid with the two parallel planes $z=h$ and $z=H ~(H>h)$, through hypergeometric function approach i.e, by using series rearrangement technique, Mellin-Barnes type contour integral representations of generalized hypergeometric function$~_pF_q(z)$ and Meijer's $G$-function; in terms of Appell's function of second kind and general triple hypergeometric series of Srivastava. The formula is neither available in the literature of mathematics nor found in any mathematical tables. We conclude that several formulas for exact curved surface area of frustum of other three dimensional figures can be derived in an analogous manner, using Mellin-Barnes contour integration. Moreover, the results deduced above (presumably new), have potential applications in the fields of applied mathematics, statistics and engineering sciences.

\section*{Disclosure statement}
No potential conflict of interest was reported by the authors.

\end{document}